\documentclass[12pt]{amsart}
\usepackage{amsmath}
\usepackage{amscd}
\usepackage{amssymb}
\usepackage{amsfonts}

\newtheorem{theorem}{Theorem}[section]

\newtheorem{corollary}[theorem]{Corollary}
\newtheorem{proposition}[theorem]{Proposition}
\newtheorem{claim}[theorem]{Claim}

\theoremstyle{definition}

\theoremstyle{remark}
\newtheorem{remark}[theorem]{Remark}
\numberwithin{equation}{section}

\begin{document}
\title[Shrinkers with curvature pinching conditions are compact]
{Shrinkers with curvature pinching conditions\\
are compact}

\author{Guoqiang Wu}
\address{School of Science, Zhejiang Sci-Tech University, Hangzhou 310018, China}
\email{gqwu@zstu.edu.cn}

\author{Jia-Yong Wu}
\address{Department of Mathematics, Shanghai University, Shanghai 200444, China;
and Newtouch Center for Mathematics of Shanghai University, Shanghai 200444, China}
\email{wujiayong@shu.edu.cn}
\thanks{}
\dedicatory{}
\date{\today}

\begin{abstract}
In this paper, we give various curvature pinching conditions such that 
shrinkers are compact. On one hand, we prove that shrinkers with positive 
Ricci curvature are compact when they have bounded curvature and certain 
curvature pinching conditions. On the other hand, we prove that shrinkers 
with certain asymptotically nonnegative sectional curvature are compact. 
As applications, some related classifications of shrinkers are provided.
\end{abstract}
\maketitle

\section{Introduction}
For an $n$-dimensional complete Riemannian manifold $(M,g)$ and a smooth
potential function $f$ on $(M,g)$, the triple $(M, g, f)$ is called a
\emph{gradient shrinking Ricci soliton} or \emph{shrinker} (see \cite{[Ham]})
if
\begin{align}\label{Eq1}
Ric+\mathrm{Hess}\,f=\frac 12g,
\end{align}
where $Ric$ is the Ricci curvature of $(M,g)$ and $\text{Hess}\,f$ is
the Hessian of $f$. On $(M, g, f)$, there always exists a minimum point
$p\in M$ of $f$, but it possibly is not unqiue; see \cite{[HM]}. Shrinkers
are viewed as a natural extension of Einstein manifolds. More importantly,
shrinkers play an important role in the Ricci flow as they correspond to
some self-similar solutions and arise as limits of dilations of Type I
singularities in the Ricci flow \cite{[Ham]}. Shrinkers can also be
regarded as critical points of the Perelman's entropy functional and play
a significant role in Perelman's resolution of the Poincar\'e conjecture
\cite{[Pe],[Pe2],[Pe3]}.

In \cite{[Pe2]}, Perelman showed that any $3$-dimensional noncollapsed shrinker
with positive and bounded sectional curvature must be compact. Later, combining
Ni-Wallach's work \cite{[NW]} and Cao-Chen-Zhu's observation \cite{[CCZ]}, one
can finish the classification of three dimensional shrinker without any assumption.
Hence one finally concludes that any $3$-dimensional shrinker with positive sectional
curvature must be compact. Based on this, Cao \cite{[Cao]} conjectured that there does
not exist any $n$-dimensional ($n\ge 4$) noncompact noncollapsed shrinker
with positive sectional curvature. In \cite{[Na]}, Naber proved that any
$4$-dimensional shrinker with positive Ricci curvature and bounded nonnegative
curvature operator must be compact. In \cite{[MW]}, Munteanu-Wang completely
answered Cao's question by exploring Chow-Lu-Yang's argument \cite{[CLY]}.
They showed that any $n$-dimensional shrinker with positive Ricci curvature
and nonnegative sectional curvature must be compact.

In this paper, we first improve Munteanu-Wang's compact result \cite{[MW]}
to the $2$nd-Ricci curvature case.  After finishing the first version
of this paper, we learned that Li-Ni had proven this result from the
viewpoint of the $PIC_1$ condition; see Theorem 3.1 in \cite{[LN]}.
Here we propose the result from the $k$th-Ricci curvature.

\begin{theorem}\label{thm1}(Li-Ni \cite{[LN]})
Let $(M,g, f)$ be an $n$-dimensional shrinker with positive Ricci curvature
and nonnegative $2$nd-Ricci curvature. Then $(M,g, f)$ must be compact.
\end{theorem}

Recall that, the \emph{$k$th-Ricci curvature} on an $n$-dimensional
Riemannian manifold $(M,g)$, satisfies $Ric_{(k)}\ge H$ for some
constant $H$, at a point $x\in M$ if for all $(k+1)$-dimensional
subspaces $V\in T_xM$, the Riemannian curvature $Rm$ of $(M,g)$
satisfies
\[
\sum^{k+1}_{i=1}Rm(v,e_i,v,e_i)\ge H
\]
for all $v\in V$, where $\{e_1, ..., e_{k+1}\}$ is any orthonormal basis
for $V$. In particular, $Ric_{(2)}\geq 0$ if $Rm(e_1, e_2, e_1, e_2)+Rm(e_1, e_3, e_1, e_3)\geq 0$ for any three orthonormal vectors $\{e_1, e_2, e_3\}$.
Obviously, this definition requires $1\le k\le n-1$ otherwise
it is vacuous. Moreover, if $Ric_{(k)}\ge H$ at all points $x\in M$, we
say $(M,g)$ satisfies $k$th-Ricci curvature bounded below by $H$, i.e.,
$Ric_{(k)}(M)\ge H$. Clearly, $Ric_{(k)}(M)=Ric(M)$ if $k=n-1$, while
$Ric_{(k)}(M)$ is the sectional curvature $K_{sec}(M)$ if $k=1$.
If $1\le k_1\le k_2\le n-1$, then $Ric_{(k_1)}(M)\ge 0$
implies $Ric_{(k_2)}(M)\ge 0$. Hence $Ric_{(k)}(M)$ can be viewed
as the interpolation between the Ricci curvature and the sectional
curvature. We refer the readers to \cite{[Wh],[Sh]} and references
therein for geometry and topology of the $k$-th Ricci curvature.

Theorem \ref{thm1} implies that any $n$-dimensional shrinker with positive
$2$nd-Ricci curvature must be compact. Moreover, Theorem \ref{thm1} implies
the following classification.

\begin{corollary}\label{cor1}
Let $(M,g, f)$ be an $n$-dimensional shrinker with nonnegative $2$nd-Ricci
curvature. Then $(M,g,f)$ must be compact, or $\mathbb{R}^n$, or a quotient
of the product $\mathbb{R}^k\times N^{n-k}$ with $1\le k\le n-2$, where
$N^{n-k}$ is a compact shrinker of dimension $n-k$ with positive Ricci curvature.
\end{corollary}

The proof of Theorem \ref{thm1} follows the argument of Munteanu-Wang
\cite{[MW]} and the main difference is that we need to carefully analyze
the nonnegativity of curvature term $Rm*Ric$ in the $f$-Laplacian
equation of Ricci curvature (see \eqref{evo} below). The idea of the proof is
as follows. We first observe that $Ric_{2}(M)\ge0$ implies $Rm*Ric\ge 0$
in \eqref{evo}. Assuming that $(M,g,f)$ is noncompact, then we can apply
$Rm*Ric\ge 0$ to show that the scalar curvature is unbounded at infinity,
which contradicts a fact that the average of scalar curvature in shrinkers
is bounded.

At present, we do not know whether $Rm*Ric\ge 0$ holds when
$Ric_{k}(M)\ge0$ for some $3\le k\le n-1$. However if we impose certain eigenvalues pinching of Ricci curvature, we still have
$Rm*Ric\ge 0$ and hence the shrinkers are compact, even for
$k=n-1$ (e.g., Theorem \ref{Sect1} in Section \ref{2thRic}). Besides,
let $\lambda_1\le \lambda_2\le\ldots\le\lambda_n$ denote eigenvalues of
Ricci curvature in $(M,g,f)$. We have that
\begin{theorem}\label{thm2}
Let $(M,g,f)$ be an $n$-dimensional shrinker with positive Ricci curvature.
Suppose that the scalar curvature is bounded when $n=4$ or the sectional
curvature is bounded when $n\ge 5$. If there exists a positive constant
$A$ such that
\begin{equation}\label{pinmis}
\lambda_n-\lambda_2\le A \lambda_1,
\end{equation}
then $(M,g, f)$ is compact.
\end{theorem}

Similarly,  Theorem \ref{thm2} implies the following classification.
\begin{corollary}\label{cor1b}
Let $(M, g, f)$ be an $n$-dimensional shrinker with nonnegative Ricci
curvature. Suppose that the scalar curvature is bounded when $n=4$ or
the sectional curvature is bounded when $n\ge 5$. If \eqref{pinmis} holds,
then $(M,g,f)$ must be compact, or $\mathbb{R}^n$, or a quotient of the
product $\mathbb{R}\times N^{n-1}$ , where $N^{n-1}$ is a compact
shrinker of dimension $n-1$ with $Ric(g_N)=\frac 12g_N$.
\end{corollary}

In particular, for dimension 4, we get the same compact conclusion
under different curvature pinching conditions.
\begin{theorem}\label{thmpicn}
Let $(M, g, f)$ be a nonflat $4$-dimensional shrinker with positive
Ricci curvature and bounded scalar curvature. For any $x\in M$, if there
exists a positive constant $A$ such that
\begin{equation}\label{picncon}
|Rm(u,v,u,v)|\le A\cdot Ric(u, u)
\end{equation}
for any orthonormal vectors $u$ and $v$ in $T_xM$, then $(M, g, f)$ is compact.
\end{theorem}
\begin{remark} Recently, Li-Wang \cite{[LW2023]} proved that any K\"ahler
shrinker surface must have bounded curvature. However, it is not known in
real four dimensional case.
\end{remark}

As an application of Theorem \ref{thmpicn}, we have

\begin{corollary}\label{cor2}
Let $(M, g, f)$ be a $4$-dimensional shrinker with bounded scalar curvature.
If \eqref{picncon} holds, then $(M,g,f)$ must be compact, or $\mathbb{R}^4$,
or a quotient of the product $\mathbb{R}^k\times\mathbb{S}^{4-k}$ with
$1\le k\le 2$, where $\mathbb{S}^{4-k}$ is the standard sphere of dimension
$4-k$ with $Ric(g_{\mathbb{S}^{4-k}})=\frac 12 g_{\mathbb{S}^{4-k}}$.
\end{corollary}

\begin{remark}
The curvature pinching condition \eqref{picncon} was introduced by Qu-Wu
\cite{[QW]}, where it was named by the so-called ``condition A". There exist
many compact shrinkers satisfying condition \eqref{picncon} but sectional
curvature changing sign; see for example the compact K\"ahler shrinker on $\mathbb{C}\mathbb{P}^2\#(-\mathbb{C}\mathbb{P}^2)$ constructed by Cao
\cite{[C96]} and Koiso \cite{[Ko]}, independently. Since the nonnegative
sectional curvature or the nonnegative $2$nd-Ricci curvature obviously
implies the condition \eqref{picncon} (see also \cite{[QW]}), our
results may be suitable to a larger class than before. We would like
to point out that under more additional curvature assumptions, the same
conclusions of Theorem \ref{thmpicn} and Corollary \ref{cor2} were proved
in \cite{[QW]}.
\end{remark}

The proof of Theorems \ref{thm2} and \ref{thmpicn} not only adapts Munteanu-Wang's
argument \cite{[MW],[MW19]}, but also uses Naber's structure theorem \cite{[Na]}.
The idea of the proof  is as follows. We first assume for a contradiction that the shrinker
is noncompact. Then by Naber's structure theorem, there exists a sequence of pointed
manifolds  such that it smoothly converges to a limit manifold
along each integral curve of $\nabla f$. Finally we can exclude all possible limit
manifolds; thus the noncompact assumption is not true. Once we get the shrinkers are compact,
our classification results (Corollaries \ref{cor1} and \ref{cor2}) then follow by
Petersen-Wylie's maximum principle \cite{[PW]} and de Rham splitting theorem.
We remark that more compact theorems of shrinkers under different
eigenvalues pinching conditions  will be discussed in Sections \ref{pinchII} and \ref{anc}.

There are many related works concerning the compact property of shrinkers.
Wu-Zhang \cite{[WZ]} applied Munteanu-Wang's argument \cite{[MW]}
to reprove Ni's  theorem \cite{[Ni]}, which states that any K\"ahler
shrinker with positive bisectional curvature must be compact. In \cite{[Zh]},
Zhang used a similar method of \cite{[MW]} to generalize Ni's compact result
to the case of positive orthogonal bisectional curvature, which also
independently obtained by Li-Ni \cite{[LN]}. In \cite{[LW]},
Li-Wang proved that if the second eigenvalues of curvature operator is
positive, then the shrinker must be a quotient of the standard sphere.
In \cite{[GLX]}, Guan-Lu-Xu obtained a compact theorem for codimension
one shrinkers in $\mathbb{R}^{n+1}$ with the positive Ricci curvature.

The rest of this paper is organized as follows. In Section \ref{pre},
we recall some basic results and asymptotic properties at infinity for
shrinkers. In Section  \ref{2thRic}, we first analyze the nonnegativity
of a key algebraic curvature term on shrinkers with $Ric_{(2)}(M)\ge 0$.
Then we can prove Theorem \ref{thm1} and Corollary \ref{cor1} by adapting
Munteanu-Wang's proof strategy. We also provide another criterion for the
compact property of shrinkers (see Theorem \ref{Sect1}). In Section
\ref{pinchI}, we will prove Theorem \ref{thm2} by analyzing the
asymptotic structure of the shrinker. Then we apply Theorem \ref{thm2}
to prove Corollary \ref{cor1b}. In Section \ref{pinchII}, we will prove
Theorem \ref{thmpicn} and Corollary \ref{cor2}. Meanwhile, we will
apply different eigenvalues pinching conditions to discuss more compact results
for shrinkers. In Section \ref{anc}, we will study compact theorems
of shrinkers with certain asymptotically nonnegative curvature.

\section{Preliminaries}\label{pre}
In this section, we shall recall some basic facts about shrinkers, which
will be used in the proof of our theorems. First, tracing \eqref{Eq1}
gives
\begin{equation}\label{tra}
S+\Delta f=\frac n2,
\end{equation}
where $S$ is the scalar curvature of $(M,g,f)$. By Hamilton's computation
\cite{[Ham]}, by adding a constant to $f$ if necessary, from equations
\eqref{Eq1} and \eqref{tra}, one has
\begin{equation}\label{equat}
S+|\nabla f|^2=f
\end{equation}
and
\[
\nabla S=2Ric(\nabla f).
\]
Combining \eqref{tra} and \eqref{equat} implies that
\begin{equation}\label{feq}
\Delta_f f=\frac n2-f,
\end{equation}
where $\Delta_f:=\Delta-\nabla f\cdot\nabla$ is the weighted Laplacian.
Moreover, Hamilton \cite{[Ham]} derived the following equations for scalar
curvature and Ricci curvature:
\begin{equation}\label{Sequat}
\Delta_f S= S-2|Ric|^2
\end{equation}
and
\begin{equation}\label{evo}
\Delta_fR_{ij}=R_{ij}-2R_{ikjl}R_{kl},
\end{equation}
where $R_{ij}$ and $R_{ikjl}$ are the Ricci curvature and the Riemannian
curvature of $(M,g,f)$.

On the other hand, Chen \cite{[Chen]} proved that $S\ge 0$ on any complete
shrinker. From \cite{[PiRS]}, we know that $S>0$, unless $(M,g,f)$ is the
Gaussian shrinker $(\mathbb{R}^n, g_E, \frac{|x|^2}{4})$. Later, Chen's
result was refined by Chow-Lu-Yang \cite{[CLY]} to be that
\[
S\ge c\,f^{-1}
\]
for some constant $c>0$ on any non-flat shrinker. This estimate is sharp and
achieved by the complete noncompact K\"ahler  shrinkers constructed by
Feldman-Ilmanen-Knopf \cite{[FIK]}.

From Cao-Zhou's work \cite{[CZ]} (see also Haslhofer-M\"uller \cite{[HM]}),
the potential function $f$ has quadratic growth at infinity. Precisely,
\begin{theorem}\label{pote}
Let $(M,g, f)$ be an $n$-dimensional complete noncompact shrinker
with \eqref{Eq1} and \eqref{equat}. Then there exists a point $p\in M $
where $f$ attains its infimum. Moreover,
\[
\frac 14\left[\big(r(x)-5n\big)_{+}\right]^2\le f(x)\le\frac 14\left(r(x)+\sqrt{2n}\right)^2,
\]
where $r(x)$ is a distance function from $p$
to $x$, and $a_+=\max\{a,\,0\}$ for $a\in \mathbb{R}$.
\end{theorem}

Combining this with \eqref{equat}, if the scalar curvature $S$ is
bounded, then there exists $r_0>0$ such that the level set
\[
\Sigma(r):=\{x\in M|f(x)=r\}
\]
of $f$ is a compact manifold for $r\ge r_0$. Meanwhile, the
domain
\[
D(r):=\{x\in M|f(x)\le r\}
\]
is also a compact manifold with boundary $\Sigma(r)$.

On  an $n$-dimensional complete noncompact shrinker $(M,g,f)$, by
\cite{[CZ],[MW12]}, there exist constants $c_1$ and $c_2$ such that
\[
c_1 r\le V(B_p(r))\le c_2r^n
\]
for any $r>1$, where $V(B_p(r))$ denotes the volume of the geodesic
ball $B_p(r)$ with radius $r$ and center at $p\in M$. Combining this
with Theorem \ref{pote}, we immediately conclude that the weighted
volume of $M$ is finite. That is,
\[
V_f(M):=\int_Me^{-f}dv<\infty.
\]
Moreover, Munteanu-Sesum \cite{[MS]} showed that
\begin{equation}\label{Ricfini}
\int_M|Ric|^2e^{-\lambda f}dv<\infty
\end{equation}
for any constant $\lambda>0$, on any shrinker $(M,g,f)$. This property will be
used in the proof of our results.

As we all know, the curvature operator on any $3$-dimensional shrinker
must be nonnegative, and hence is bounded by the scalar curvature. For
dimension $4$ and higher, the curvature operator no longer has a fixed sign;
see Feldman-Ilmanen-Knopf's example in \cite{[FIK]}. But for dimension $4$,
Munteanu-Wang \cite{[MWa]} showed that the curvature operator still can be
controlled by the bounded scalar curvature. Namely,
\begin{theorem}\label{RmdyS}
Let $(M, g, f)$ be a $4$-dimensional shrinker with scalar curvature
$S\le c_1$ for some constant $c_1>0$. Then there exists a constant $c>0$
such that
\[
|Rm|\le c\,S
\]
on $(M, g, f)$. Here the constant $c$ depends only on $c_1$ and the geometry
of the geodesic ball $B_p(r_0)$,
where $p$ is a minimum point of $f$ and $r_0$ is determined by $c_1$.
\end{theorem}

In another direction, Naber \cite{[Na]} applied the singular reduced
length functions of Perelman to get the geometry at infinity
of a shrinker. This result will be repeatedly used in our proof.
\begin{theorem}\label{infi}
Let $(M, g, f)$ be an $n$-dimensional shrinker with bounded curvature
operator. Then for all sequences $x_i\in M$ going to infinity along an
integral curve of $\nabla f$, there exists a subsequence also denoted
by $x_i$, such that $(M,g,x_i)$ smoothly converges to a product manifold
$\mathbb{R}\times N$, where $N$ is an $(n-1)$-dimensional shrinker.
\end{theorem}
In Theorem \ref{infi}, we do not know whether the limit manifold
$\mathbb{R}\times N$ depends on the choice of sequence $x_i$. When $N$
is a quotient of the round sphere, Munteanu-Wang \cite{[MW19]} improved
Naber's result. More precisely, they proved that if the round cylinder
$\mathbb{R}\times \mathbb{S}^{n-1}/\Gamma$ occurs as a limit for a
sequence of points going to infinity along an end of the shrinker,
then the end is smoothly asymptotic to the same round cylinder.

\section{Shrinkers without curvature upper bounds}\label{2thRic}
In this section, we mainly study compact theorems of shrinkers without upper bounds
of curvature. From the viewpoint of the $k$th-Ricci curvature, we prove Theorem
\ref{thm1} by analyzing the nonnegativity of $2Rm\ast Ric$ in \eqref{evo}, which
appears in Munteanu-Wang's argument \cite{[MW]}. Meanwhile we apply Theorem
\ref{thm1} to prove Corollary \ref{cor1}. We also give another compact theorem
of shrinkers under certain curvature lower assumption.

Let $\lambda_1\le\lambda_2\le\ldots\le \lambda_n$ be eigenvalues of
Ricci curvature $Ric$. Indeed, at a point $x\in M$, let $e_1$ be
an eigenvector with respect to the minimal eigenvalues $\lambda_1$.
Then we extend $e_1$ to an orthonormal basis $\{e_1,e_2,\ldots, e_n\}$
such that $\{e_i\}^n_{i=1}$ are the eigenvectors of $Ric$ with respect
to the corresponding eigenvalues $\{\lambda_i\}^n_{i=1}$. By \eqref{evo},
we have
\begin{equation}\label{evoRic}
\Delta_f \lambda_1\le \lambda_1-2Rm(e_1,e_i,e_1,e_j)Ric(e_i,e_j)
\end{equation}
in the barrier sense.
Diagonalizing the Ricci curvature $Ric$ such that
$R_{kl}:=Ric(e_k,e_l)=\lambda_k\delta_{kl}$. Then,
\begin{align*}
I&:=2Rm(e_1,e_i,e_1,e_j)Ric(e_i,e_j)\\
&=2(K_{12}\lambda_2+K_{13}\lambda_3+\ldots+K_{1n}\lambda_n),
\end{align*}
where $K_{ij}$ denotes the sectional curvature of the plane spanned by $e_i$
and $e_j$. Notice that, we also have
\begin{align*}
\lambda_1&:=R_{11}=K_{12}+K_{13}+\ldots+ K_{1n},\\
\lambda_2&:=R_{22}=K_{21}+K_{23}+\ldots+ K_{2n},\\
&\quad\cdots\cdots\\
\lambda_{n-1}&:=R_{(n-1)(n-1)}=K_{(n-1)1}+K_{(n-1)2}+\ldots+K_{(n-1)(n-2)}+K_{(n-1)n},\\
\lambda_n&:=R_{nn}=K_{n1}+K_{n2}+\ldots+K_{n(n-1)}.
\end{align*}

Below, we will show the following claim, which will be used in the proof of
Theorem \ref{thm1}.
\begin{claim}\label{cla}
If $Ric_{(2)}(M)\ge 0$, then
\[
I:=2(K_{12}\lambda_2+K_{13}\lambda_3+\ldots+K_{1n}\lambda_n)\ge 0.
\]
\end{claim}
\begin{proof}[Proof of Claim \ref{cla}]
We prove the claim by three steps according to the sign of $K_{1i}$,
where $i=1,2,\ldots, n$. Since $Ric_{(2)}(M)\ge 0$, then we have only
one $K_{1i}\le 0$ for some $2\le i\le n$.

\emph{Step 1}: when $K_{12}\le 0$, from the condition $Ric_{(2)}(M)\ge 0$,
we see that $K_{13}, \ldots, K_{1n}$ are all nonnegative.
Using $K_{12}\le 0$ and $\lambda_2\le\lambda_n$, we get
\begin{align*}
I&:=2K_{12}\lambda_2+2K_{13}\lambda_3+\ldots+2K_{1n}\lambda_n\\
&\ge 2K_{12}\lambda_n+2K_{13}\lambda_3+\ldots+2K_{1n}\lambda_n\\
&=2K_{13}\lambda_3+\ldots+2K_{1(n-1)}\lambda_{n-1}+2(K_{1n}+K_{12})\lambda_n.
\end{align*}
Since $Ric_{(2)}(M)\ge 0$, then
\begin{align*}
2\lambda_3&=2K_{31}+2K_{32}+2K_{34}+\ldots+ 2K_{3n}\\
&=(K_{31}+K_{3n})+(K_{32}+K_{3(n-1)})+\ldots+(K_{3(n-1)}+K_{32})+(K_{3n}+K_{31})\ge 0,\\
&\quad\cdots\cdots\\
2\lambda_{n-1}&=2K_{(n-1)1}+2K_{(n-1)2}+\ldots+2K_{(n-1)(n-2)}+2K_{(n-1)n}\\
&=(K_{(n-1)1}+K_{(n-1)n})+(K_{(n-1)2}+K_{(n-1)(n-2)})+\ldots+2(K_{(n-1)(n-2)}+K_{(n-1)2})\\
&\quad+(K_{(n-1)n}+K_{(n-1)1})\ge 0,\\
2\lambda_n&=2K_{n1}+2K_{n2}+\ldots+2K_{n(n-1)}\\
&=(K_{n1}+K_{n(n-1)})+(K_{n2}+K_{n(n-2)})+\ldots+(K_{n(n-1)}+K_{n1})\ge0,
\end{align*}
and
\[
K_{1n}+K_{12}\ge 0.
\]
Putting these together, we conclude that $I\ge 0$ on $M$.

\emph{Step 2}: when $K_{1i}\le 0$, where $3\le i\le n-1$, the argument
is similar to Step 1. Indeed, from $Ric_{(2)}(M)\ge 0$  we get that
$K_{12}, \ldots, K_{1(i-1)}, K_{1(i+1)}, \ldots, K_{1n}$ are all nonnegative.
Then using $K_{1i}\le 0$ and $\lambda_i\le \lambda_n$, we have that
\begin{align*}
I&:=2K_{12}\lambda_2+2K_{13}\lambda_3+\ldots+2K_{1n}\lambda_n\\
&\ge 2K_{12}\lambda_2+\ldots+2K_{1i}\lambda_n+\ldots+2K_{1n}\lambda_n\\
&=2K_{12}\lambda_2+\ldots+2K_{1(i-1)}\lambda_{i-1}+2K_{1(i+1)}\lambda_{i+1}+\ldots
+2(K_{1n}+K_{1i})\lambda_n\\
&\ge 0,
\end{align*}
where in the last inequality, we used
$\lambda_2,\ldots,\lambda_{i-1},\lambda_{i+1},\ldots,\lambda_n$
are all nonnegative and $K_{1n}+K_{1i}\ge0$ by a similar argument of Step 1,
due to the condition $Ric_{(2)}(M)\ge 0$. So, $I\ge 0$.

\emph{Step 3}: when $K_{1n}\le 0$. This case is a little different from
the above two cases. From the curvature condition $Ric_{(2)}(M)\ge 0$, we see that
$K_{12}, K_{13}, \ldots, K_{1(n-1)}$ are all nonnegative and
$K_{n2}, K_{n3}, \ldots, K_{n(n-1)}$ are also all nonnegative.
Therefore,
\begin{align*}
I&:=2K_{12}\lambda_2+2K_{13}\lambda_3+\ldots+2K_{1n}\lambda_n\\
&=K_{12}(2K_{21}+2K_{23}+\ldots+ 2K_{2n}),\\
&\quad+K_{13}(2K_{31}+2K_{32}+2K_{34}+\ldots+ 2K_{3n})\\
&\quad\cdots\cdots\\
&\quad+K_{1n}(2K_{n1}+2K_{n2}+\ldots+2K_{n(n-1)})\\
&=K_{12}(2K_{21}+2K_{23}+\ldots+ 2K_{2(n-1)}),\\
&\quad+K_{13}(2K_{31}+2K_{32}+2K_{34}+\ldots+ 2K_{3(n-1)})\\
&\quad\cdots\cdots\\
&\quad+2K_{1n}^2+2K_{2n}(K_{1n}+K_{2n})+2K_{3n}(K_{1n}+K_{3n})+2K_{(n-1)n}(K_{1n}+K_{(n-1)n}),
\end{align*}
where in the last equality, we moved the terms $2K_{12}K_{2n}, 2K_{13}K_{3n},\ldots,2K_{1(n-1)}K_{n(n-1)}$
into the last line above. To go on checking our claim, let
\begin{align*}
I_2&:=K_{12}(2K_{21}+2K_{23}+\ldots+ 2K_{2(n-1)}),\\
I_3&:=K_{13}(2K_{31}+2K_{32}+2K_{34}+\ldots+ 2K_{3(n-1)}),\\
&\quad\cdots\cdots\\
I_n&:=2K_{1n}^2+2K_{2n}(K_{1n}+K_{2n})+2K_{3n}(K_{1n}+K_{3n})+2K_{(n-1)n}(K_{1n}+K_{(n-1)n}).
\end{align*}
Then $I$ can be written as
\[
I=I_2+I_3+\ldots+I_n.
\]
Since $K_{12}, K_{13}, \ldots, K_{1(n-1)}\ge 0$, $K_{n2},K_{n3}, \ldots, K_{n(n-1)}\ge0$
and $Ric_{(2)}(M)\ge 0$, we easily see that $I_2,I_3,\ldots I_n$ are all nonnegative.
Hence $I\ge 0$.

In summary, if $Ric_{(2)}(M)\ge 0$, we prove $I\ge 0$.
\end{proof}

\begin{remark}\label{addcla}
Similar to the proof of Claim \ref{cla}, for any fixed $1\le k\le n$,
we indeed can show that $2Rm(e_k,e_i,e_k,e_j)Ric(e_i,e_j)\ge 0$.
\end{remark}

We are now in position to prove Theorem \ref{thm1} by Claim \ref{cla}.

\begin{proof}[Proof of Theorem \ref{thm1}]
Similar to the argument of \cite{[MW]}, the theorem is proved by
contradiction. We sketch the proof and the details can be referred
to \cite{[MW]}. Assume that $(M,g,f)$ is complete noncompact, and let
$p$ be a minimum point of $f$.

Since $Ric>0$ on $(M,g,f)$,  we may let $\lambda_1=\lambda_1(x)>0$
denote the minimal eigenvalue of Ricci curvature at $x$. Combining
\eqref{evoRic} and Claim \ref{cla}, we have
\[
\Delta_f\lambda_1\le\lambda_1
\]
on shrinker, in the sense of barriers. For a sufficiently large $R_1$
depending only on $n$, let $a:=\min_{\partial B_p(R_1)}\lambda_1>0$
and then the function
\[
u:=\lambda_1-a f^{-1}-na f^{-2}
\]
satisfies
\[
u>0\quad \text{on} \quad\partial B_p(R_1).
\]
Using \eqref{feq}, we have
\begin{equation}\label{basic}
\Delta_f(f^{-1})\ge f^{-1}-\frac n2f^{-2}
\quad \text{and}\quad
\Delta_f(f^{-2})\ge \frac 32f^{-2}
\end{equation}
on $M\setminus B_p(R_1)$. Then $\Delta_fu\le u$ on $M\setminus B_p(R_1)$.
By the maximum principle, we can show that $u\ge 0$
on $M\setminus B_p(R_1)$ and hence there exists some constant $0<b<1$
such that
\[
Ric\ge bf^{-1}\quad\text{on}\quad M.
\]
Following the same argument of \cite{[MW]}, that is, applying the above
Ricci curvature estimate to the integral curve argument, by some basic
properties of shrinkers, we are able to prove that
\[
S\ge b\ln f
\]
on $M\setminus B_p(R_1)$. In particular, we have
\[
S\ge n
\]
on $M\setminus B_p(R_2)$ for some sufficiently large $R_2>R_1$.

On the other hand, recall that by \cite{[CZ]}, for all $r>0$, we have
\[
\int_{B_p(r)}S dv\le\frac{n}{2}V(B_p(r)).
\]
Since $S\ge n$ on $M\setminus B_p(R_2)$, then for any $r>R_2$,
\[
n\Big[V(B_p(r))-V(B_p(R_2))\Big]\le\int_{B_p(r)\backslash B_p(R_2)}S dv\le\frac{n}{2}V(B_p(r)).
\]
Since $c_1 r\le V(B_p(r))\le c_2r^n$ for any $r>1$, the above inequality
does not hold when $r$ is large enough.
\end{proof}

Next, Corollary \ref{cor1} easily follows from Theorem \ref{thm1}.
\begin{proof}[Proof of Corollary \ref{cor1}]
In Corollary 4 of \cite{[PW]}, Petersen and Wylie showed that the universal cover
of shrinker $(M,g,f)$ satisfying  $Rm(u, e_k, u, e_l)Ric(e_k, e_l)\geq 0$ for any $u$  and
$|Ric|\in L^2(e^{-f}dv)$ must be isometric to  $\mathbb{R}^k\times N$,
where $N$ has positive Ricci curvature. In our setting, since we assume
$Ric_{(2)}(M)\ge 0$, by Claim \ref{cla} and Remark \ref{addcla}, we conclude
that $Rm(u, e_k, u, e_l)Ric(e_k, e_l)\geq 0$ for any $u$. Moreover, by \eqref{Ricfini},
we see that $|Ric|\in L^2(e^{-f}dv)$ automatically holds for all shrinkers.
Hence the corollary follows.
\end{proof}

Finally we will provide another lower bounds of sectional curvature such
that shrinkers are compact. Recall that let
$\lambda_1\le \lambda_2\le\ldots\le\lambda_n$ denote eigenvalues of
Ricci curvature in $n$-dimensional shrinker $(M,g,f)$. At a point
$x\in M$, let $\{e_1,e_2,\ldots, e_n\}$ be an orthonormal basis at $x$
such that $\{e_i\}^n_{i=1}$ are the eigenvectors of $Ric$ with respect
to the corresponding eigenvalues $\{\lambda_i\}^n_{i=1}$.
Set
\[
K_{e_1}:=\inf_{2\le i\le n}{K_{1i}},
\]
where $K_{ij}$ denotes the sectional curvature of the plane spanned
by $e_i$ and $e_j$. The following result is also a generalization
of Munteanu-Wang's result \cite{[MW]}.
\begin{theorem}\label{Sect1}
Let $(M,g, f)$ be an $n$-dimensional shrinker with positive Ricci curvature.
If
\begin{equation}\label{Sectcond1}
(\lambda_i-\lambda_2)K_{e_1}\ge-\frac{\lambda_1\lambda_2}{n-2},\quad i=3,\cdots, n,
\end{equation}
for any $x\in M$, then $(M,g, f)$ must be compact.
\end{theorem}

\begin{proof}[Proof of Theorem \ref{Sect1}]
The proof is very similar to the case of Theorem \ref{thm1}, and we only need
to check $\sum^n_{i=2}K_{1i}\lambda_i\ge 0$. Indeed, under our curvature
assumptions,
\begin{align*}
\sum^n_{i=2}K_{1i}\lambda_i&=\sum^n_{i=2}K_{1i}(\lambda_i-\lambda_2)+\lambda_1\lambda_2\\
&\ge\sum^n_{i=3}-\frac{\lambda_1\lambda_2}{n-2}+\lambda_1\lambda_2\\
&=0
\end{align*}
and the theorem follows by the preceding argument.
\end{proof}

%%%%%%%%%%%%%%%%%%%%%%%%%%%%%%%%%%%%%%%%%%%%%%%%%%%%%%%%%%%%%%%%%%%%%%%%%%%%%%%%%%%%%%%%%%%%%%%%%%%%%%%%%%%%%%%%%%%%%%%%%%%%%%%%%%%%%%%%%%%%%%
%%%%%%%%%%%%%%%%%%%%%%%%%%%%%%%%%%%%%%%%%%%%%%%%%%%%%%%%%%%%%%%%%%%%%%%%%%%%%%%%%%%%%%%%%%%%%%%%%%%%%%%%%%%%%%%%%%%%%%%%%%%%%%%%%%%%%%%%%%%%%%

\section{Shrinkers with  eigenvalues pinching condition I}\label{pinchI}

In this section, we will apply the arguments of \cite{[MW],[MW19],[Na]}
to study the compact property of shrinkers with positive Ricci curvature.
We mainly prove Theorem \ref{thm2} and Corollary \ref{cor1b} in the introduction.

\begin{proof}[Proof of Theorem \ref{thm2}]
Suppose by contradiction that $(M^n,g,f)$ is complete and noncompact.
We remark that our curvature assumption of theorem implies Riemannian curvature
is bounded. Because on $4$-dimensional shrinkers, the bounded scalar curvature
implies the bounded Riemmanian curvature; see Theorem \ref{RmdyS}.
We first prove the following claim.
\begin{claim}\label{cla2}
For a complete noncompact shrinker $(M^n,g,f)$ with $Ric>0$ and bounded curvature,
the minimal eigenvalue $\lambda_1(x)$ of Ricci curvature tends to zero uniformly at infinity.
\end{claim}
If Claim \ref{cla2} is false, then there exists a sequence of points
$p_i$ tending to infinity such that
\[
\lambda_1(p_i)\ge \delta
\]
for some constant $\delta>0$. From Theorem 1.4 in \cite{[Na]}, we know that
for the shrinker $(M^n,g,f)$ with bounded curvature, the associated
Ricci flow $(M^n, g(t), p_i)$ defined on $(-\infty, 1)$, where
\[
g(t):=(1-t)\phi^*_tg,\quad \frac{d\phi_t}{dt}=\frac{\nabla f}{1-t}
\quad  \text{and}\quad \phi_0=\mathrm{Id},
\]
is $\kappa$-noncollapsed, where $\kappa=\kappa(n, V_f(M))$. By Hamilton's
Cheeger-Gromov compactness theorem \cite{[Ham]}, the associated Ricci flow
sub-converges to an ancient and $\kappa$-noncollapsed solution
$(M^n_{\infty}, g_{\infty}(t), p_{\infty})$ . Since the curvature tensor is
bounded according to our assumption, by \eqref{equat} and Theorem \ref{pote},
we have that
\begin{equation}\label{lininfty}
|\nabla f(x)|\to \infty
\end{equation}
at the linear growth rate as $x\to \infty$. Now we consider a sequence of functions
\[
f_i(x):=\frac{f(x)-f(p_i)}{|\nabla f(p_i)|}.
\]
Then it satisfies
\[
|\nabla f_i(p_i)|=1 \quad\text{and} \quad
|\nabla^2 f_i|=\frac{|\tfrac{1}{2}g-Ric(g)|}{|\nabla f(p_i)|}.
\]
on $(M^n,g(t))$. Moreover, by \eqref{lininfty}, we know that
$|\nabla f(p_i)|\to\infty$ as $p_i\to\infty$. Combining this with the bounded
curvature assumption, we conclude that $|\nabla^2 f_i|\to 0$ uniformly as
$i\to\infty$. Hence along the convergence of $(M^n, g(t), p_i)$, the sequence
of functions $f_i(x)$ smoothly converges to a limit $f_\infty$ satisfying
\[
|\nabla f_\infty(p_{\infty})|=1\quad\text{and}\quad
\nabla^2 f_\infty\equiv 0
\]
on $(M^n_{\infty}, g_{\infty}(t))$. It is easy to see that
$(M^n_{\infty}, g_{\infty}(t), p_{\infty})$ isometrically splits
as $\mathbb{R}\times N^{n-1}$, where $N^{n-1}$ is an $(n-1)$-dimensional
ancient and $\kappa$-noncollapsed solution with nonnegative Ricci curvature.
In particular, we have $\lambda_1(p_{\infty})=0$. This contradicts
$\lambda_1(p_i)\ge \delta$ for some constant $\delta>0$ and hence
Claim \ref{cla2} is true.

We now continue to prove Theorem \ref{thm2}. By Theorem \ref{infi}, we know that
for any sequence of points $x_i\in M$ going to infinity along an integral curve
of $\nabla f$, $(M^n,g,x_i)$ smoothly converges to $\mathbb{R}\times N^{n-1}$,
where $N^{n-1}$ is another shrinker. According to our eigenvalues pinching condition  and
$Ric>0$, we have that the limit manifold $N^{n-1}$ must be an $(n-1)$-dimensional
nonflat Einstein manifold. Indeed we easily see that $\lambda_1=0$, $\lambda_i=1/2$
for any $2\le i\le n$ and $S=\frac{n-1}{2}$ on $\mathbb{R}\times N^{n-1}$.

On the other hand, since $S$ is bounded, by \eqref{equat} and Theorem \ref{pote},
we can choose some sufficiently large $R_1>0$ such that $|\nabla f|\neq 0$ on $M\setminus\overset{\circ}{D}(R_1)$.

Now for any point $y_0\in \Sigma(R_1)$, considering an integral curve
$\gamma_{y_0}(t)$, where $t\ge 0$, to $\nabla f$ with $\gamma_{y_0}(0)=y_0$,
the scalar curvature $S$ along the integral curve $\gamma_{y_0}(t)$ satisfies
\[
\frac{dS}{dt}(\gamma_{y_0}(t))=\langle\nabla S,\nabla f\rangle=2Ric(\nabla f,\nabla f)>0,
\]
on $M\setminus \overset{\circ}{D}(R_1)$, where we used $Ric>0$ and
$\nabla S=2Ric(\nabla f)$. So $S$ increases along each integral curve
of $\nabla f$ outside a compact set of $M^n$. This implies that
\[
S\le \frac{n-1}{2}\quad \text{and}\quad
\lambda_2\le\frac{1}{2}
\]
outside a compact set of $M^n$. Therefore, our   eigenvalues pinching
assumption $\lambda_n-\lambda_2\le A \lambda_1$ of theorem could imply
$\lambda_i-1/2\le A \lambda_1$ for any $2\le i\le n$ outside a compact
set of $M^n$.

From the above discussion, it is obvious that we can apply new  eigenvalues pinching
$(\lambda_i-1/2)\le A\lambda_1$ for any $2\leq i\le n$ outside a compact set
of $M^n$ (instead of the assumption of theorem $\lambda_n-\lambda_2\le A \lambda_1$)
to prove our theorem. We directly compute that
\begin{align*}
\frac 12\Delta_f S&=\frac 12S-|Ric|^2\\
&=\frac 12\sum^n_{i=1}\lambda_i-\sum^n_{i=1}\lambda_i^2\\
&=\sum^n_{i=1}\left(\frac 12-\lambda_i\right)\lambda_i\\
&=-\sum^n_{i=2}\left(\frac 12-\lambda_i\right)^2+\frac 12\sum^n_{i=2}\left(\frac 12-\lambda_i\right)
+\left(\frac 12-\lambda_1\right)\lambda_1\\
&\ge-(n-1)A^2 \lambda_1^2 +\frac 12\left(\frac{n-1}{2}-\sum_{i=2}^n\lambda_i\right)
+\left(\frac 12-\lambda_1\right)\lambda_1
\end{align*}
outside a compact set of $M^n$, where we used  eigenvalues pinching
$(\lambda_i-1/2)\le A \lambda_1$ in the last inequality. Since
$\lambda_1(x)$ converges to zero uniformly at infinity and
$S\le\frac{n-1}{2}$ outside a compact set of $M^n$, then
\[
-(n-1)A^2 \lambda_1^2\ge-\frac{1}{4}\lambda_1,\quad
\frac{n-1}{2}-\sum_{i=2}^n\lambda_i\ge\lambda_1\quad\text{and}\quad
\frac 12-\lambda_1>\frac 13
\]
outside a sufficiently large compact set of $M^n$. Substituting these estimates into
the above inequality we finally get
\begin{equation}\label{Sext}
\frac 12\Delta_f S\ge-\frac{1}{4}\lambda_1+\frac 12\lambda_1+\frac 13\lambda_1>0
\end{equation}
outside a sufficiently large compact set of $M^n$.

On the other hand, we may assume $K=D(a)$  for some sufficiently large $a$
such that $\nabla f\neq 0$ on $M\setminus D(a)$. As in \cite{[MW19]},
on $M\setminus D(a)$, using \eqref{Sext}, the Stokes theorem and $V_f(M)<\infty$,
we have
\begin{align*}
0&<\int_{M\setminus D(a)}\Delta_f S\cdot e^{-f}dv\\
&=-\int_{\Sigma(a)}\langle\nabla S, \tfrac{\nabla f}{|\nabla f|}\rangle\cdot e^{-f}dv\\
&=-2\int_{\Sigma(a)}Ric\left(\nabla f, \tfrac{\nabla f}{|\nabla f|}\right)\cdot e^{-f}dv\\
&< 0.
\end{align*}
This is a contradiction and hence $(M,g,f)$ is compact.
\end{proof}

\begin{remark}
We would like to point out that Theorem \ref{thm2} still holds under a
weaker assumption
\[
(\lambda_n-\lambda_2)^{\delta}\le A \lambda_1
\]
for some constant $1\le \delta<2$. Since the proof is almost the same as the
Theorem \ref{thm2}, we omit the proof here.
\end{remark}

Next, we will apply Theorem \ref{thm2} to get the following classification
which was stated as Corollary \ref{cor1b} in the introduction.
\begin{proof}[Proof of Corollary \ref{cor1b}]
Without loss of generality we can assume $(M, g, f)$ is simply connected.
By \eqref{evoRic} and   eigenvalues pinching \eqref{pinmis}, we have
\begin{align*}
\Delta_f \lambda_1&\le\lambda_1-2\sum_{i=2}^nK_{1i}\lambda_i\\
&=\lambda_1-2\sum_{i=2}^n K_{1i}\cdot\lambda_2-2\sum_{i=2}^n K_{1i}(\lambda_i-\lambda_2)\\
&\le\lambda_1-2\lambda_1 \lambda_2+ C A \lambda_1\\
&\le C' A \lambda_1
\end{align*}
in the barrier sense, where $C$ and $C'$ are constants depending only on the
curvature bound. Since the Ricci curvature is nonnegative, by
the strong maximum principle to the above inequality, we only have the following
two cases.

\emph{Case 1}: $\lambda_1>0$. Then $(M, g)$ is compact by Theorem \ref{thm2}.

\emph{Case 2}: $\lambda_1\equiv0$. Then there exists two possibilities
according to our  eigenvalues pinching condition. One is $\lambda_i=0$ for any $2\le i\le n$,
the other is $\lambda_i=\frac{1}{2}$ for any $2\le i\le n$. The former
case is of course $\mathbb{R}^n$. The latter case is $\mathbb{R}\times N^{n-1}$,
where $N^{n-1}$ is an $(n-1)$-dimensional Einstein manifold with Einstein constant
$\frac{1}{2}$; see Corollary 2 in \cite{[PW]}.

Putting these cases together, the corollary follows.
\end{proof}

\section{Shrinkers with   eigenvalues pinching condition II}\label{pinchII}
In this section, adapting the argument of preceding sections, we continue to discuss
some pinching conditions such that shrinkers with positive Ricci curvature are compact.
In particular, we will prove Theorem \ref{thmpicn} and Corollary \ref{cor2} in the
introduction.

Recall that, in \cite{[MW19]}, Munteanu-Wang proved any real $4$-dimensional
K\"ahler shrinker with positive bounded Ricci curvature must be compact. Here
we shall give a generalization using   eigenvalues pinching condition. Denoted by
$\lambda_1\le\lambda_2\le\lambda_3\le\lambda_4$ eigenvalues of Ricci
curvature in a $4$-dimensional shrinker $(M,g,f)$.
\begin{theorem}\label{twopin}
For any $4$-dimensional shrinker $(M,g,f)$ with positive Ricci curvature
and bounded scalar curvature, if there exists a positive constant $A$
such that
\[
\lambda_2\le A \lambda_1
\quad\text{and}\quad
\lambda_4-\lambda_3\le A \lambda_1,
\]
then $(M, g, f)$ is compact.
\end{theorem}
\begin{remark}
For any real $4$-dimensional K\"ahler shrinker with positive Ricci curvature,
the above eigenvalues pinching condition automatically holds, because
we have $\lambda_1=\lambda_2$ and $\lambda_3=\lambda_4$ on such shrinker.
\end{remark}

\begin{proof}[Proof of Theorem \ref{twopin}]
Assume now by contradiction that the shrinker $(M, g, f)$ is complete
and noncompact with bounded scalar curvature. Similar to the argument
of Theorem \ref{thm2}, we still have that $\lambda_1$ tends to zero
uniformly at infinity. Then there are two possible cases for the
structure at infinity of $(M, g, f)$.

\emph{Case 1}: $(M^4, g, f)$ converges to the limit manifold
$\mathbb{R}\times N^3$ along each integral curve of $\nabla f$, where
$N^3$ is a compact shrinker and hence $\mathbb{S}^3/\Gamma$. It
is impossible due to our   eigenvalues  pinching $\lambda_2\le A \lambda_1$.

\emph{Case 2}: $(M^4, g, f)$ converges to $\mathbb{R}\times N^3$
along each integral curve of $\nabla f$, where $N^3$ is a noncompact
shrinker and hence $N^3$ is a quotient of $\mathbb{R}\times \mathbb{S}^2$.
Meanwhile the scalar curvature $S$ converges to $1$ at infinity.
Moreover, we have
\[
\lambda_3=\lambda_4=\frac 12
\]
on $\mathbb{R}\times N^3$. Since $Ric>0$, from the proceeding discussion, we know that
$S$ increases along an integral curve of $\nabla f$ outside a compact
set of $M$. Hence there exists a compact set $K\subset M$ so that $S\le 1$
on $M\setminus K$, where $K$ contains all critical points of $f$.

Also, using the conditions $0<\lambda_1\le\lambda_2\le\lambda_3\le\lambda_4$ and
$\lambda_4-\lambda_3\le A \lambda_1$, on $M\setminus K$, by \eqref{Sequat}
we compute that
\begin{align*}
\Delta_f S&= S-2|Ric|^2\\
&\geq S^2-2 |Ric|^2\\
&=(\lambda_1+\lambda_2+\lambda_3+\lambda_4)^2-2(\lambda_1^2+\lambda_2^2+\lambda_3^2+\lambda_4^2)\\
&=   2\lambda_1\lambda_2+ 2\lambda_1\lambda_3+2\lambda_1\lambda_4+2\lambda_2\lambda_3+2\lambda_2\lambda_4 +2\lambda_3\lambda_4-(\lambda_1^2+\lambda_2^2+\lambda_3^2+\lambda_4^2)\\
&\ge 2 \lambda_1\lambda_4-(\lambda_4-\lambda_3)^2\\
&\ge 2\lambda_1\lambda_4-A^2 \lambda_1^2.
\end{align*}
In the above computation, we can assume that $\lambda_4\ge\frac 13$ and
$A^2\lambda_1\le\frac 14$ on $M\setminus K$. Hence we further get
\[
\Delta_f S\ge\left(\frac 23-\frac 14\right)\lambda_1=\frac{5}{12}\lambda_1>0
\]
on $M\setminus K$. The rest of the proof is the same as the case of Theorem \ref{thm2}.
\end{proof}

In the rest of this section, we will prove Theorem \ref{thmpicn} and Corollary
\ref{cor2} in the introduction. We remark that Qu-Wu \cite{[QW]} proved the
same results under more assumption condition.

\begin{proof}[Proof of Theorem \ref{thmpicn}]
The theorem is proved by contradiction. We assume that the shrinker $(M^4, g, f)$
is complete and noncompact. Let $p\in M$ be a minimum point of $f$.
Similar to the argument of Section \ref{2thRic}, we have
\begin{equation}\label{eigenevo}
\Delta_f \lambda_1\le \lambda_1-2(K_{12}\lambda_2+K_{13}\lambda_3+K_{14}\lambda_4)
\end{equation}
in the sense of barriers, where
\begin{equation}
\begin{aligned}\label{eigendef}
\lambda_1&:=K_{12}+K_{13}+K_{14},\\
\lambda_2&:=K_{21}+K_{23}+K_{24},\\
\lambda_3&:=K_{31}+K_{32}+K_{34},\\
\lambda_4&:=K_{41}+K_{42}+K_{43}.
\end{aligned}
\end{equation}
Here $0<\lambda_1\le \lambda_2\le\lambda_3\le \lambda_4$ are eigenvalues of
positive Ricci curvature $R_{ij}$ and $K_{ij}$ denotes the sectional
curvature of the plane spanned by $e_i$ and $e_j$.

For the $4$-dimensional shrinker, by Theorem \ref{RmdyS}, the bounded
scalar curvature implies the Riemmanian curvature is bounded. Since
$Ric>0$, the shrinker $(M,g,f)$ has only one end. Therefore we can apply
Theorem \ref{infi} to conclude that there exists a sequence of points
$x_i\in M$ tending to infinity along an integral curve of $\nabla f$, such
that $(M,g,x_i)$ smoothly converges to $\mathbb{R}\times N^3$, where $N^3$ is
a $3$-dimensional shrinker.

\emph{Case 1}: $N^3=\mathbb{S}^3/\Gamma$. In this case, Munteanu-Wang \cite{[MW19]}
proved that $(M, g)$ converges to $\mathbb{R}\times \mathbb{S}^3/\Gamma$
uniformly. So if $d(x, p)>r_0$ for some sufficiently large constant $r_0>0$,
then we must have
\begin{equation}\label{estim}
A\big(|\varepsilon_2|+|\varepsilon_3|+|\varepsilon_4|\big)<\frac{1}{2},
\end{equation}
where $\varepsilon_i:=\lambda_i-\frac{1}{2}$ for $i=2,3,4$. Then we have
the following claim.
\begin{claim}\label{cla3}
$2(K_{12}\lambda_2+K_{13}\lambda_3+K_{14}\lambda_4)>0$
outside a sufficiently large compact set of $M^4$.
\end{claim}

To confirm this claim, at $x\in M\setminus \overline{B(p, r_0)}$,
under the curvature assumption \eqref{picncon} and \eqref{estim},
we obtain
\begin{align*}
2(K_{12}\lambda_2+K_{13}\lambda_3+K_{14})\lambda_4
&=2K_{12}\left(\frac{1}{2}+\varepsilon_2\right)
+2K_{13}\left(\frac{1}{2}+\varepsilon_3\right)
+2K_{14}\left(\frac{1}{2}+\varepsilon_4\right)\\
&=\lambda_1+2 K_{12}\varepsilon_2+2 K_{13}\varepsilon_3+2 K_{14}\varepsilon_4\\
&\ge\lambda_1-2 A\cdot\lambda_1\big(|\varepsilon_2|+|\varepsilon_3|+|\varepsilon_4|\big)\\
&=\lambda_1\left[1-2 A\cdot\big(|\varepsilon_2|+|\varepsilon_3|+|\varepsilon_4|\big)\right]\\
&>0,
\end{align*}
which proves Claim \ref{cla3}.

Combining \eqref{eigenevo} with Claim \ref{cla3} yields
\[
\Delta_f\lambda_1\le \lambda_1
\]
in $M\setminus \overline{B(p, r_0)}$, in the sense of barriers. In the rest
we can follow the argument of Theorem \ref{thm1} and prove $M^4$ is compact.
This contradicts our assumption.

\emph{Case 2}: $N^3=(\mathbb{R}\times \mathbb{S}^2)/\Gamma$. In this case,
by \eqref{eigendef} and the assumption \eqref{picncon}, we have
\begin{align*}
\lambda_4-\lambda_3&=(K_{41}-K_{31})+(K_{42}-K_{32})\\
&\le A \lambda_2+A \lambda_2\\
&=2 A \lambda_2.
\end{align*}
This condition is the same as Case 2 of Theorem \ref{twopin}.
We can completely follow the argument of Case 2 of Theorem \ref{twopin}
to get that $M^4$ is compact, which is also a contradiction.
Indeed, the condition $\lambda_2\le A \lambda_1$ of Theorem \ref{twopin}
is unnecessary in the proof of the compact theorem for the case
$N^3=(\mathbb{R}\times \mathbb{S}^2)/\Gamma$.
\end{proof}

In the following, similar to the argument of Corollary \ref{cor1b},
we will apply Theorem \ref{thmpicn} to prove Corollary \ref{cor2}
in the introduction.
\begin{proof}[Proof of Corollary \ref{cor2}]
We may assume $(M^4, g, f)$ is simply connected without loss of generality.
By \eqref{evoRic} and    curvature pinching \eqref{picncon}, we have
\begin{align*}
\Delta_f \lambda_1&\le\lambda_1-2(K_{12}\lambda_2+K_{13}\lambda_3+K_{14}\lambda_4)\\
&\le\lambda_1+2A\lambda_1\lambda_2+2A\lambda_1\lambda_3+2A\lambda_1\lambda_4\\
&=\lambda_1+2A\lambda_1(S-\lambda_1)\\
&\le C\lambda_1
\end{align*}
in the barrier sense, where $C$ is a constant depending on $A$ and the
scalar curvature bound. Since the Ricci curvature is nonnegative, we apply
the strong maximum principle to the above inequality and conclude that
$\lambda_1>0$ or $\lambda_1\equiv0$.

\emph{Case 1}: $\lambda_1>0$. In this case, $(M, g)$ is compact by Theorem
\ref{thmpicn}.

\emph{Case 2}: $\lambda_1\equiv0$. By the maximum principle of \cite{[PW]}
and the de Rham splitting theorem, we have $M^4=\mathbb{R}\times N^3$, where
$N^3$ is a shrinker. Since any complete $3$-dimensional shrinkers have been
classified (see {[CCZ]}), then $N^3$ is isometric to $\mathbb{S}^3$,
$\mathbb{R}\times\mathbb{S}^2$, $\mathbb{R}^3$.

To sum up, Corollary \ref{cor2} follows.
\end{proof}

\section{Shrinkers with asymptotically nonnegative curvature}\label{anc}

In this section, we will discuss the compact property of shrinkers with certain
asymptotically nonnegative curvature operator. We observe that if more
restriction of Ricci curvature is imposed, we are able to prove the shrinkers are compact under some asymptotically nonnegative sectional curvature assumption.

Use the same notations in Section \ref{2thRic}. Recall that
\[
K_{e_1}:=\inf_{2\le i\le n}{K_{1i}},
\]
where $K_{ij}$ denotes the sectional curvature of the plane spanned
by $e_i$ and $e_j$.
\begin{theorem}\label{asynoRic}
For any $n$-dimensional shrinker $(M,g,f)$, if
\[
S\cdot K_{e_1}\ge-\frac{c_1}{f^2}\quad\text{and}\quad
Ric\ge\frac{c_2}{f^{2-\epsilon}}
\]
for any $x\in M$, where $0<\epsilon<1$, $c_1$ and $c_2$
are fixed positive constants, then $(M,g,f)$ is compact.
\end{theorem}

\begin{remark}
Munteanu-Wang \cite{[MW]} showed that any shrinker with positive Ricci curvature
and nonnegative sectional curvature must be compact. Theorem \ref{asynoRic} allows
sectional curvature to be asymptotically nonnegative and hence it generalizes
their result in some sense.
\end{remark}

\begin{proof}[Proof of Theorem \ref{asynoRic}]
The proof is contradiction. Assume $(M, g, f)$ is complete noncompact
and it satisfies $S\cdot K_{e_1}\ge-\frac{c_1}{f^2}$ and
$Ric\ge\frac{c_2}{f^{2-\epsilon}}$. Combining these constrictions and
\eqref{evoRic}, we have
\begin{equation}
\begin{aligned}\label{extratwo}
\Delta_f\lambda_1&\le\lambda_1-2(K_{12}\lambda_2+K_{13}\lambda_3+\ldots+K_{1n}\lambda_n)\\
&\le\lambda_1+\frac{2c_1}{S\cdot f^2}(\lambda_2+\lambda_3+\ldots+\lambda_n)\\
&=\lambda_1+\frac{2c_1}{S\cdot f^2}(S-\lambda_1)\\
&\le\lambda_1+2c_1\,f^{-2}
\end{aligned}
\end{equation}
in the barrier sense. Consider the function
\[
u:=\lambda_1-af^{-1}-2nc_1 f^{-2},
\]
where $a:=\inf_{\partial B_p(r_0)}\lambda_1>0$. Here $p\in M$ is a minimum point of $f$
and $r_0>0$ is a large number determined later. On $M\setminus B_p(r_0)$,
by \eqref{basic} and \eqref{extratwo}, we have
\begin{equation*}
\begin{aligned}
\Delta_fu&\le\lambda_1+2af^{-2}-af^{-1}+\frac n2 af^{-2}-3 nc_1f^{-2}\\
&=u+\left(2c_1+\frac{n}{2}a-nc_1\right)f^{-2}
\end{aligned}
\end{equation*}
on $M\setminus B_p(r_0)$. Choosing a sufficiently large $r_0$, we have
\[
2c_1+\frac{n}{2}a-nc_1<0
\]
because $a\to 0+$ as $r_0\to\infty$. Thus,
$\Delta_fu\le u$ on $M\setminus B_p(r_0)$.

On the other hand, since $Ric\ge\frac{c_2}{f^{2-\epsilon}}$,
by the definition of $u$, we may choose $r_0$ large enough such that
\begin{align*}
u&\ge a(1-f^{-1})-2nc_1 f^{-2}\\
&\ge \frac{c_2}{f^{2-\epsilon}}\cdot\frac{1}{2}-2nc_1 f^{-2}\\
&>0
\end{align*}
on $\partial B_p(r_0)$, where $0<\epsilon<1$.

In conclusion, for a sufficiently large $r_0>0$, we show that
\[
\Delta_fu\le u\quad \text{on} \quad M\setminus B_p(r_0),
\]
and
\[
u>0 \quad \text{on} \quad \partial B_p(r_0).
\]
Then the rest of our argument is
the same as the proof of Theorem \ref{thm1}.
\end{proof}

At the end of this section, we give an asymptotical estimate of
certain curvature quantity for complete noncompact shrinkers.
\begin{proposition}\label{asynon}
For any $n$-dimensional complete noncompact shrinker $(M,g,f)$ with positive
Ricci curvature, there exists a point $x_0\in M$ such that
\[
S(x_0)\cdot K_{e_1}(x_0)<-a\,r^{-4}(x_0),
\]
where $r(x_0)$ is a distance function from point $p\in M$ (a minimum point of $f$)
to $x_0$, and $a:=\inf_{\partial B_p(2022n)}\lambda_1$.
\end{proposition}
\begin{remark}
In Theorem \ref{asynon}, we only provide a fixed radius
$2022n$, which is of course not best and it is possibly improved to a smaller one.
\end{remark}

\begin{proof}[Proof of Proposition \ref{asynon}]
The proof is by contradiction. Assume that there exists a complete noncompact
shrinker $(M, g, f)$ with $Ric>0$ satisfying
\[
S(x)\cdot K_{e_1}(x)\ge-a\,r^{-4}(x)
\]
for all $x\in M$, where $a:=\inf_{\partial B_p(2022n)}\lambda_1>0$.
Combining these constrictions and \eqref{evoRic}, we have that
\begin{equation}
\begin{aligned}\label{extra}
\Delta_f\lambda_1&\le\lambda_1-2(K_{12}\lambda_2+K_{13}\lambda_3+\ldots+K_{1n}\lambda_n)\\
&\le\lambda_1+\frac{2a}{S\cdot r^4(x)}(\lambda_2+\lambda_3+\ldots+\lambda_n)\\
&=\lambda_1+\frac{2a}{S\cdot r^4(x)}(S-\lambda_1)\\
&\le\lambda_1+2a\,r^{-4}(x)
\end{aligned}
\end{equation}
in the barrier sense. Consider the function
\[
u:=\lambda_1-af^{-1}-2022n^2a f^{-2}.
\]
Since $a:=\inf_{\partial B_p(2022n)}\lambda_1>0$, then
\[
u>0\quad\text{on}\quad \partial B_p(2022n),
\]
where we used Theorem \ref{pote}:
\[
\frac 14\big(2022n-5n\big)^2\le f(x)
\le\frac 14\left(2022n+\sqrt{2n}\right)^2
\quad\text{on}\quad \partial B_p(2022n).
\]
On $M\setminus B_p(2022n)$, by \eqref{basic} and \eqref{extra}, we get
\begin{equation*}
\begin{aligned}
\Delta_fu&\le\lambda_1+2ar^{-4}(x)-af^{-1}+\frac{na}{2}f^{-2}-3033n^2af^{-2}\\
&=u+2ar^{-4}(x)-\frac{n}{2}(2022n-1)af^{-2}
\end{aligned}
\end{equation*}
on $M\setminus B_p(2022n)$. Moreover, Theorem \ref{pote} gives
\[
\frac 14(r(x)-5n)^2\le f(x)\le\frac 14(r(x)+\sqrt{2n})^2
\]
on $M\setminus B_p(2022n)$, and so
\[
2ar^{-4}(x)-\frac{n}{2}(2022n-1)af^{-2}\le 0.
\]
Therefore we show that
\[
\Delta_fu\le u\quad \text{on} \quad M\setminus B_p(2022n),
\]
with $u>0$ on $\partial B_p(2022n)$.
In the rest, we follow the argument as Theorem \ref{thm1}
and prove that $M$ is compact. This contradicts our noncompact
assumption.
\end{proof}

\textbf{Acknowledgement}.
The authors thank Xiaolong Li for making them aware of the work of
his joint work \cite{[LN]}. The first author would like to thank
Jianyu Ou and Yuanyuan Qu for helpful discussion.

%%%%%%%%%%%%%%%%%%%%%%%%%%%%%%%%%%%%%%%%%%%%%%%%%%%%%%%%%%%%%%%%%%%%%%%%%%%%%%%%%%%%%%%%%%%%%%%%%%%%%%%%%%%%%%%%%%%%%%%%%%%%%%%%%%%%%%%%%%%%%%
%%%%%%%%%%%%%%%%%%%%%%%%%%%%%%%%%%%%%%%%%%%%%%%%%%%%%%%%%%%%%%%%%%%%%%%%%%%%%%%%%%%%%%%%%%%%%%%%%%%%%%%%%%%%%%%%%%%%%%%%%%%%%%%%%%%%%%%%%%%%%%
% ------------------------------------------------------------------------
\bibliographystyle{amsplain}

\end{document}